\documentclass[12pt]{amsart}
\usepackage{amssymb,mathrsfs,graphicx,latexsym}
\usepackage{fullpage}
\AtEndDocument{\vfill\eject\batchmode}
\newtheorem{thm}{Theorem}
\newtheorem{cor}{Corollary}[section]
\newtheorem{lemma}{Lemma}[section]
\newtheorem{prop}{Proposition}[section]

\theoremstyle{definition}
\newtheorem{defn}{Definition}[section]

\newtheorem{remark}{Remark}[section]
\newtheorem{example}{Example}[section]
\newtheorem{con}{Conjecture}
\makeatletter
\DeclareSymbolFont{script}{U}{eus}{m}{n}
\DeclareMathSymbol{\Wedge}{0}{script}{"5E}
\DeclareSymbolFont{rmslops}{OT1}{cmr}{m}{sl}
\DeclareSymbolFontAlphabet{\mathrmsl}{rmslops}
\def\operator@font{\mathgroup\symrmslops}
\makeatother

\newcommand{\liealg}[2]{\newcommand{#1}{{\mathfrak{#2}}}}
\newcommand{\C}{\ensuremath{\mathbb{C}}}
\newcommand{\R}{\ensuremath{\mathbb{R}}}
\newcommand{\CP}{\ensuremath{\mathbb{CP}}}

\newcommand{\Z}{\ensuremath{\mathbb{Z}}}

\renewcommand{\H}{\ensuremath{\mathbb{H}}}
\renewcommand{\O}{\ensuremath{\mathbb{O}}} 

\liealg\gl{gl}\liealg\sgl{sl}\liealg\symp{sp}
\liealg\g{g}\liealg\p{p}
\newcommand{\bdm}{\begin{displaymath}}
\newcommand{\edm}{\end{displaymath}}
\def\haken{\mathbin{\hbox to 6pt{%
                 \vrule height0.4pt width5pt depth0pt
                 \kern-.4pt
                 \vrule height6pt width0.4pt depth0pt\hss}}}

\DeclareMathOperator{\End}{End}

\DeclareMathOperator{\spanu}{span}

\DeclareMathOperator{\id}{id}

\newcommand{\ra}[1]{{\raise6pt\hbox{$#1$}}}
\begin{document}
\title[On nearly  K\"ahler $6$-manifolds]{$S^6$ and the geometry of nearly  K\"ahler $6$-manifolds}

\author[I. Agricola, A.W. Bor\'owka]{Ilka Agricola, Aleksandra Bor\'owka, 
Thomas Friedrich}
\address{\hspace{-5mm} 
Ilka Agricola\newline
Fachbereich Mathematik und Informatik \newline
Philipps-Universit\"at Marburg\newline
Hans-Meerwein-Strasse \newline
D-35032 Marburg, Germany\newline
{\normalfont\ttfamily agricola@mathematik.uni-marburg.de}}
\address{\hspace{-5mm} 
Aleksandra Bor\'owka\newline
Faculty of Mathematics and Computer Science\newline
Jagiellonian University\newline
ul. Lojasiewicza 6\newline 
30-348 Krakow, Poland\newline
{\normalfont\ttfamily aleksandra.borowka@uj.edu.pl}}
\address{\hspace{-5mm} 
Thomas Friedrich\newline
Institut f\"ur Mathematik \newline
Humboldt-Universit\"at zu Berlin\newline
Sitz: WBC Adlershof\newline
D-10099 Berlin, Germany\newline
{\normalfont\ttfamily friedric@mathematik.hu-berlin.de}}
%
\begin{abstract}
We review results on and around the almost complex structure on $S^6$, both 
from a classical and a modern point of view. These notes have been 
prepared for the Workshop ``(Non)-existence of complex structures on $S^6$'' (\emph{Erste Marburger Arbeitsgemeinschaft Mathematik -- MAM-1}), held in Marburg in March 2017.
\end{abstract}
\maketitle
%
\section{Introduction}
%
It is well known that the sphere $S^6$ admits an almost Hermitian structure induced by octonionic 
multiplication, and that this structure stems from the transitive action of the compact
exceptional Lie group $G_2$ on it.
In 1955, Fukami and Ishihara were presumably the first authors to devote a 
separate paper to the detailed investigation of $S^6$  and showed in particular that $S^6$ is the naturally 
reductive space $G_2/\mathrm{SU}(3)$ \cite{FI55}. In 1958, Calabi studied hypersurfaces in the 
space of imaginary octonions and proved that the induced  almost complex structure is never
integrable if the hypersurface is compact \cite{Cal58}. In fact the almost Hermitian structure on the 
$6$-sphere is a very special one: Already in \cite{FI55}, it is observed that
the Levi Civita derivative of $J$ satisfies 
\bdm\tag{$*$}
(\nabla^g_XJ)X\ =\ 0 \text{ for all vector fields }X. 
\edm
Such manifolds are called \emph{nearly K\"ahler} and they were investigated  intensively 
by Gray in a series of papers \cite{Gray66,Gray70,G76}.  In particular, 
he showed in  dimension $6$ that they are Einstein and their first Chern class vanishes.
In fact, for many reasons  dimension 
$6$ is of particular interest for nearly K\"ahler geometry \cite{Nagy10}. For a long time the only 
compact examples of nearly 
K\"ahler manifolds were the four homogeneous examples: $S^6$, $S^3\times S^3$, $\CP^3$ and 
the flag manifold $F_2$. 
The aim of this paper is to provide a concise review of properties of nearly K\"ahler manifolds in 
dimension $6$ with special attention given to the sphere $S^6$. After some historical remarks,
we start by recalling Calabi's 
result about hypersurfaces in the space of imaginary octonions $\R^7$. Then we discuss the 
intrinsic torsion approach and naturally reductive spaces and briefly recall Gray's \cite{G76} 
and Kirichenko's \cite{Kir77} results. 
Next we 
present the spinoral approach of R. Grunewald \cite{Gru90} and a modern view on it.
We finish by giving an 
overview of L. Foscolo and M. Haskins contribution \cite{FH17}. They discovered
non-homogeneous cohomogeneity one 
nearly K\"ahler structures on $S^6$ and conjectured that these are the only cohomogeneity one examples.

\section{Some historical comments}
Clearing the facts around the almost complex structure on $S^6$ took several independent steps.
In particular, it was not noticed immediately that (and how) 
it was related to the transitive action of $G_2$.

Montgomery and Samelson proved in 1943 that the only compact connected simple Lie group which can be 
transitive on $S^{2n}$ is $\mathrm{SO}(2n+1)$---\emph{except for a a finite number of $n's$}
\cite[Thm II, p.462]{Montgomery&S43}.  Their method was of topological nature and required 
the knowledge of the
homology rings of simple Lie groups, which was not yet available for the  five exceptional simple
Lie groups; hence they couldn't give any further information on the exceptional cases.

Six years later, Armand Borel proceeded by constructing the homogeneous spaces directly, which 
lead him to the result that the only sphere with a transitive group $G$ acting that is not 
orthogonal is $S^6$ with  $G=G_2$ \cite[Thm III, p. 586]{Borel49}. This completed the classification
of transitive sphere actions and showed, in particular, that $G_2$ is the only exceptional Lie group with
such an action.

Meanwhile, Adrian Kirchhoff had noticed in 1947 that $S^6$ carries an almost complex structure
induced from octonionic multiplication (\cite{Kirchhoff47}; see also \cite{Ehresmann50}). 
In his main theorem, Kirchhoff related 
the existence of an almost complex structure on $S^n$ to the parallelism of $S^{n+1}$. 
 
In 1951,  Ehresmann and Libermann \cite{EL51} as well as Eckmann and Fr\"olicher \cite{Eckmann&F51}
observed independently that this almost complex structure on $S^6$ is not integrable --- in fact, their 
articles appeared in the same volume of the Comptes Rendus Hebdomadaires des S\'eances de 
l'Acad\'emie des Sciences, Paris. While Eckmann and Fr\"olicher were  interested in formulating the
integrability condition and treated $S^6$ merely as an example where it didn't hold, 
the aim of Ehresmann and Libermann was the local description 
of locally homogeneous almost hermitian manifolds in terms of Cartan structural equations, and they 
found that the equations exhibited an exceptional structure for  $n=6$. They stated \cite[p. 1282]{EL51}:
 
\begin{quote}
``La structure consid\'er\'ee est donc localement \'equivalente \`a une structure presque hermitienne
sur $S^6$ admettant $G_2$ comme groupe d'automorphismes. Ce groupe ne peut laisser
invariante sur $S^6$ qu'une seule structure presque complexe. Celle-ci est donc
isomorphe \`a la structure presque complexe d\'efinie \`a l'aide des octaves de
Cayley. Comme la deuxi\`eme torsion dans les formules (5) n'est pas nulle,
cette structure ne d\'erive pas d'une structure complexe.''\footnote{
 ``The structure we considered is therefore locally equivalent to an almost hermitian
 structure on $S^6$ admitting $G_2$ as its group of automorphisms. This group can
 leave invariant only one almost complex structure on $S^6$. It is therefore isomorphic
 to the almost complex structure defined with the help of Cayley's octonions. Since the
 second torsion of the formulas (5) doesn't vanish, this structure is not induced from a complex
 structure.'' (translated by the authors)}
\end{quote}

Hence, they seem to be the first authors to connect the transitive $G_2$-action on $S^6$ to 
its octonionic almost  complex structure.
A detailed account of the results of \cite{Eckmann&F51} and further material was given by 
Fr\"olicher four years later \cite{Froelicher55}; however, in his discussion of
homogeneous  almost complex manifolds, he doesn't mention $S^6$. Remarkably, he described already
(as did \cite{FI55}) the characteristic connection of $S^6$ and proved that its torsion is given by the 
Nijenhuis tensor.

In \cite{FI55}, all these thoughts on $S^6$ are brought together for the first time, and the characteristic 
connection is proved to coincide with the canonical connection of the homogeneous 
space $G_2/\mathrm{SU}(3)$.

The first author to suggest the investigation of manifolds satisfying the abstract nearly K\"ahler condition $(*)$ was Tachibana in \cite{Tachibana59},
who called such manifolds $K$-spaces and proved, amongst other things, that their Nijenhuis tensor is totally
antisymmetric. No examples were discussed, although it is clear from the reference made to
\cite{FI55, Froelicher55} that the inspiration came from $S^6$. The paper \cite{Koto60} continued
the investigation of $K$-spaces.

Inspired by the papers of Calabi \cite{Cal58} and Koto \cite{Koto60},
Alfred Gray used in 1966 for the first time the term \emph{nearly K\"ahler manifold} \cite{Gray66}.
He writes in the introduction: 
\begin{quote}
``The manifolds we discuss include complex and almost K\"ahler manifolds; 
also $S^6$ with the almost complex structure derived from the Cayley
numbers falls into a class of manifolds which we call nearly K\"ahlerian.'' 
\end{quote}
This was the starting point of the career of $S^6$ as a most remarkable nearly K\"ahler manifold.
Surprisingly, most classes of almost Hermitian manifolds that were systemized later in the
Gray-Hervella classification \cite{GH80} appear already in this paper.

\section{The almost complex structures induced from octonions}
%
In this section we present an explicit approach for constructing the nearly K\"ahler structure on $S^6$. 
The construction goes back to Calabi \cite{Cal58}, who studied hypersurfaces in $\R^7$ with a 
complex structure induced from octonions.
\subsection{Seven-dimensional cross products} 
Recall that the octonion algebra $\O$ is the unique $8$-dimensional composition algebra (or equivalently 
normed division algebra). It can be defined from quaternions using the Cayley-Dickson construction
$\O =\H \oplus J\H$, with the following operations: 
\bdm 
\overline{q_1+Jq_2}=\overline{q_1}-Jq_2,\quad (q_1+Jq_2)(q_3+Jq_4)
=q_1q_3+\overline{q_4}q_2+J(q_2\overline{q_3}+q_4q_1).
\edm
Consequently, the octonions can be viewed as an $8$ dimensional (non-associative) algebra with  basis 
$1,e_1,\ldots, e_7$, where $e_i$ are imaginary units and the multiplication between them is defined above. 
Note that we can take $e_1, e_2, e_3$ to be imaginary quaternions. 
Consider the vector subspace of imaginary octonions $\mathcal{Y}:=\spanu \{e_1,\ldots, e_7\}$ 
which, as a vector space, is isomorphic to $\R^7$. 
The octonion multiplication induces (by restriction and projection) a cross product on 
$\mathcal{Y}$ via the formula
\bdm
A\times B:=\frac{1}{2}(AB \, - \, BA) \, .
\edm
\noindent
The group of automorphisms of 
$\mathcal{Y}$ is the exceptional group $G_2$. We strongly recommend the
article   by Cristina Draper in this volume for a very thorough description of $G_2$ and its
relation to octonions, cross products, and spinors \cite{Dr17}. 
The vector space $\mathcal{Y}$ together with the cross product and the scalar product is sometimes 
called the Cayley space. The cross product has the following properties:
\begin{itemize}
\item  $\langle A , (B\times C)\rangle = \langle(A\times B) , C \rangle=:(ABC)$ (\emph{scalar triple product
identity}),
\item $A\times(A\times B)=-|A|^2B+ \langle A , B \rangle A$ (\emph{Lagrange} or \emph{Malcev identity}),
\item If C is orthogonal to $A$ and $B$, then
\bdm
(A \times B) \times C = A \times (B \times C) - \langle A , B \rangle C \, .
\edm
\end{itemize}
\noindent
However, there are some significant differences between dimensions $3$ and $7$ coming from the
non-associativity of octonions---for example, the Jacobi identity does not hold in dimension $7$.
\subsection{Almost complex structure on hypersurfaces of the Cayley space}
The results  described in this section are mainly due to  Calabi  \cite{Cal58}. 
Let $S$ be a $6$-dimensional oriented manifold immersed into the Cayley space $\mathcal{Y}$.
The canonical orientation on $\mathcal{Y}$ induces a normal vector field on $S$, called $N$. 
Consider its second fundamental form and denote by $K$ its  shape operator.
The eigenvalues of $K$ are just the  principal curvatures. We define $J\in \End (TS)$ by 
\bdm
J(X)\ :=\ N \times X  , \quad X \in TS  ,  
\edm
and take $g$ to be the metric 
on $S$ induced by the pull back of the scalar product on $\mathcal{Y}$.
\begin{lemma}
$J$ is an almost complex structure on $S$ such that $(S,J,g)$ is an almost Hermitian manifold,
and it satisfies the identity 
\bdm
\langle (\nabla^g_X J)(Y) , Z \rangle \ = \ \langle K(X) \times Y , Z \rangle .
\edm
\end{lemma}
\begin{proof}
First we need to show that $J^2=-\id$. By the Malcev identity,  
$N\times(N\times X)=-|N|^2X+\langle N , X \rangle N,$ for any $X\in (TS)$. As $X$ is perpendicular to $N$,
this  proves the claim. Then, using the scalar triple product and 
Malcev identities we obtain
\bdm
\langle J(X), J(Y) \rangle= \langle N\times X , N\times Y \rangle
= \langle N\times X)\times N , Y \rangle
= - \langle N\times (N\times X) , Y \rangle = \langle X , Y \rangle \ ,
\edm
showing the $J$-invariance of the metric. Finally,
let us compute $(\nabla^g_X J)(Y)$ for any tangent vectors $X,Y$:
\begin{eqnarray*}
(\nabla^g_X J)(Y) &=& \nabla^g_X(J(Y)) - J(\nabla^g_XY) = \nabla^g_X(N \times Y) - N \times \nabla^g_X Y \\
&=& \nabla^g_X N \times Y + N \times \nabla^g_X Y - N \times \nabla^g_X Y
= K(X) \times Y \ .
\end{eqnarray*}
This yields the claimed formula for  $\nabla^g_X J$.
\end{proof}
\begin{defn}
Let $(M,g,J)$ be an almost Hermitian manifold. 
If $(\nabla^g_XJ)X=0$ and $\nabla^g J \neq 0$, then $(M,g,J)$ is called nearly K\"ahler manifold.
\end{defn}
\noindent
The almost complex structure of a nearly K\"ahler manifold is never integrable, see
\cite{G76} or \cite{AFS05}. In fact, an easy calculation shows that its Nijenhuis tensor is
given by $N(X,Y)=4 (\nabla^g_XJ)JY$.
A direct consequence of the last lemma is the following fact.
\begin{prop}
Let $g$ be the standard metric on $S^6$ and $J$ the almost complex structure defined by the 
cross product on the Cayley space. Then $(S^6,g,J)$ is a nearly K\"ahler manifold.
\end{prop}
\noindent
%
%
\begin{prop}[Calabi \cite{Cal58}]
If $(S,I)$ is a compact, oriented $6$-manifold with an immersion into the Cayley space $\mathcal{Y}$, 
the induced almost complex structure $J$  on $S$ is non-integrable.
\end{prop}
\noindent
To prove this theorem Calabi studied the shape operator $K$ of the hypersurface. He found that the 
integrability condition for $J$ is that $K$ is complex anti-linear,
\bdm
K \circ J = - J \circ K \, .
\edm
In any closed hypersurface of the Euclidian space, there 
exists an open subset on which the second fundamental form is positive or negative definite. 
But if the tangent vector $X$ is an eigenvector
of the shape operator with eigenvalue $\lambda$ , then $J(X)$ is again
an eigenvector with eigenvalue $- \lambda$. This yields a contradiction. 
\section{$S^6$ as naturally reductive space}
%
A homogeneous Riemannian space $M=G/H$ is called reductive if 
there exists an $\mathrm{Ad}(H)$-invariant subspace $\mathfrak{m}$ such that
$\mathfrak{g}=\mathfrak{h}\oplus\mathfrak{m}$. Denote by $\langle .. \, , \, .. \rangle$
the inner product in $\mathfrak{m}$ defining the $G$-invariant metric.
If it satisfies
\bdm\tag{$**$}
\langle [X ,  Y]_{\mathfrak{m} }  ,  Z  \rangle \ = \ -  \langle [X  ,  Z]_{\mathfrak{m}} ,  Y  \rangle ,
\edm
the homogeneous space is called naturally reductive. In this case the tensor
\bdm
T(X , Y, Z) \ := \ - \langle [X  ,  Y]_{\mathfrak{m}}  ,  Z  \rangle 
\edm
is totally skew symmetric, i.e. $T$ is a $3$-form. The canonical connection
$\nabla^{c}$ of $G/H$ is the unique metric connection with skew symmetric torsion tensor $T$,
\bdm
\nabla^{c} \ = \ \nabla^g \ + \ \frac{1}{2} T \, .
\edm
The holonomy of $\nabla^{c}$ is contained in the isotropy group $H$, i.e.
the canonical connection is much more adapted to the space $G/H$ then
the Levi-Civita connection $\nabla^g$. A naturally reductive space with vanishing
torsion $T$ is a Riemannian symmetric space. 
Naturally reductive homogeneous spaces have the special property that the torsion 
and the curvature are parallel with 
respect to the canonical connection, 
\bdm
\nabla^c \mathcal{R}^c \ = \ 0 \ , \quad \nabla^c T^c \ = \ 0 \, .
\edm
Observe that a homogenous
Riemannian manifold can be naturally reductive in different ways, it depends
on the choice of the subgroup $G \subset \mathrm{Iso}(M)$ of the isometry group. However, this happens only 
for spheres or Lie groups, see the recent results of C. Olmos and S. Reggiani \cite{OR12}. 
\vspace{3mm}

\noindent
Any nearly K\"ahler manifold admits a unique hermitian connection with skew symmetric torsion, too. 
This connection has been introduced by A. Gray and is called the characteristic connection 
of the nearly K\"ahler manifold, see \cite{G76}. Let us explain the proof as well as the 
formula for the characteristic torsion.
\begin{prop}
Let $(M,g,J)$ be a nearly K\"ahler manifold. There exists a unique
metric connection preserving the almost complex structure and with skew symmetric
torsion, and its torsion $3$-form is given by the formula
\bdm
T^c(X,Y,Z) \ = \ \langle (\nabla^g_X J)(JY) \, , \, Z \rangle \, .
\edm
\end{prop}
\begin {proof}
%
Consider a metric connection $\nabla = \nabla^g + \frac{1}{2} T$ with
an arbitrary skew symmetric torsion. The condition $\nabla J = $ yields the equation
\bdm
0 \ = \ \langle (\nabla^g_X J)(Y) \, , \, Z \rangle \, + \, \frac{1}{2} T(X,JY,Z) \,
+ \frac{1}{2} T(X,Y,JZ) \, .
\edm 
Symmetrizing the latter equation with respect to $X$ and $Y$, we obtain
\bdm
0 \ = \ \langle (\nabla^g_X J)(Y) \, , \, Z \rangle \, + \, \langle (\nabla^g_Y J)(X) \, , \, Z \rangle 
\, + \, \frac{1}{2} T(X,JY,Z) \, + \, \frac{1}{2} T(Y,JX,Z) \, . 
\edm
In case the almost complex structure is nearly K\"ahler we obtain the condition
\bdm
T(X,JY,Z) \ = \ - T(Y,JX, Z),
\edm
and moreover
\bdm
T(X,Y,JZ) \ = \ T(Y,JZ,X) \ = \ - T(Z,JY,X) \ = \ T(X,JY,Z) \, .
\edm
Inserting the latter formula into the first one, we finally obtain the
formula for the characteristic torsion
\bdm
0 \ =  \ \langle (\nabla^g_X J)(Y) \, , \, Z \rangle \, + \ T(X,JY,Z) \, .
\edm
\end{proof}

If a nearly K\"ahler manifold is a reductive homogeneous space,  the canonical connection in the sense 
of  reductive spaces coincides with the characteristic connection in the sense of 
nearly K\"ahler manifolds.
With respect to the full isometry group
$G = \mathrm{SO}(7)$, the round sphere $S^6$ becomes a symmetric space,
\bdm
S^6 \ = \ \mathrm{SO}(7)/\mathrm{SO}(6),
\edm 
the canonical connection coincides with the Levi-Civita connection and is hence torsion free.
On the other side,  taking into account the almost complex structure
 $J$ induced by the octonions, $S^6$ becomes a naturally reductive space (see \cite{FI55}, \cite{R93}). Indeed, by construction $J$ and $g$ are invariant under the action of 
the automorphism group of octonions. In particular, the exceptional group $G_2
\subset SO(7)$ preserves the metric as well as the almost complex structure. Moreover, $G_2$ acts 
transitively on $S^6$. The isotropy subgroup preserves the linear complex
structure of the tangent space. Consequently, it is an $8$-dimensional
subgroup of $\mathrm{U}(3)$  isomorphic to $\mathrm{SU}(3)$, see \cite[Prop.\,5.2]{Dr17}.
\begin{prop}
$S^6 = G_2/\mathrm{SU}(3)$ is naturally reductive,  its canonical connection coincides with the
characteristic connection, and its torsion $3$-form is given by the formula
\bdm
T^c(X,Y,Z) \ = \ - \langle J(X \times Y),Z \rangle \ = \ - \langle N, (X\times Y) \times Z \rangle \, . 
\edm
\end{prop}
\begin{proof}
An explicit description of the reductive space $S^6 = G_2/\mathrm{SU}(3)$ is given in 
\cite[Remarks\,5.3, 6.4]{Dr17}. In particular, it is proved that the metric satisfies the condition
$(**)$ for being naturally reductive, and that the torsion of the canonical connection is given by the 
formula stated in the proposition. We now prove that the same formula describes the
characteristic connection. Consider the metric connection $\nabla$ defined by the formula
\bdm
\langle \nabla_XY , Z \rangle \ := \ \langle \nabla^g_XY , Z \rangle  \, - \, \frac{1}{2} \langle J(X\times Y) , Z \rangle \ . 
\edm
We compute the covariant derivative $\nabla J$ :
\bdm
\langle (\nabla_XJ)Y,Z \rangle \ = \ \langle (\nabla^g_XJ)Y,Z \rangle \,
- \frac{1}{2} \langle J(X\times Y) , Z\rangle \, + \, \frac{1}{2} 
\langle J\circ J(X \times Y) , Z \rangle \ .
\edm
Next we apply the formula for the covariant derivative $\nabla^gJ$. Then we obtain
\bdm
2 \, \langle (\nabla_XJ)Y,Z \rangle \ = \ \langle X \times Y , Z \rangle \, - \
\langle J(X \times JY) , Z \rangle \ .
\edm
Since $N$ and $X,Y$ are orthogonal, the sume of the right side vanishes. Indeed, we have
\bdm
\langle J(X\times JY), Z \rangle) \ = \ \langle N\times (X 
\times (N\times Y)), Z\rangle \ = \ - \langle N \times (N\times (X\times Y)), Z \rangle \ = \ \langle (X\times Y , Z \rangle \ .
\edm
The computation proves that the connection $\nabla$ is a metric connection preserving the almost complex structure $J$. Moreover, its torsion
\bdm
T^c(X,Y,Z) \ = \ - \langle J(X\times Y),Z) \ = \ - \langle N, (X\times Y)\times Z \rangle 
\edm
is skew symmetric. Consequently, $\nabla$ is the canonical (characteristic) connection of the
naturally reductive and nearly K\"ahler space $G_2/\mathrm{SU}(3)$.
\end{proof}
\noindent
By a theorem of Butruille \cite{B05}, no other nearly K\"ahler structure on $S^6$ can be homogeneous. 
In fact he proved that the only homogeneous nearly  K\"ahler $6$-manifolds are $S^6$, $S^3\times S^3$, 
$\CP^3$ and the flag manifold $F_2$ with their standard metrics.
\section{$G$-structures and connections with skew symmetric torsion}\label{sec2}
%
The connection defined above can be described from the point of view of $G$-structures. Let $G$ be a closed Lie subgroup of $SO(n)$ and let 
$\mathfrak{so}(n)=\mathfrak{g}\oplus\mathfrak{m}$ be the corresponding orthogonal decomposition 
of the Lie algebra $\mathfrak{so}(n)$. Then a G-structure on  Riemannian manifold $M$ is a 
reduction $\mathcal{R}$ of its frame bundle, which is  principal $SO(n)$-bundle, to the 
subgroup $G$. As the Levi-Civita connection is a $1$-form with values in $\mathfrak{so}(n)$, 
using the decomposition $\mathfrak{so}(n)=\mathfrak{g}\oplus\mathfrak{m}$,  we get a direct 
sum decomposition of its restriction to $\mathcal{R}$  into a connection in principal 
$G$-bundle $\mathcal{R}$ and a term $\Gamma$ corresponding to $\mathfrak{m}$. $\Gamma$ is a 
$1$-form on $M$ with values in the associated bundle $\mathcal{R}\times_G\mathfrak{m}$ and is 
called the intrinsic torsion. It measures the integrability of $G$-structure; the structure 
is integrable if and only if $\Gamma=0$. At a fixed point, $\Gamma$ is an element of the $G$-representation $\R^n \otimes \mathfrak{m}$.
Moreover,  one can show  that in any case when $G$ is the isotropy group of some tensor $\mathcal{T}$ the algebraic 
$G$-types of $\Gamma$ correspond to algebraic $G$-types of $\nabla^g \mathcal{T}$ (see \cite{F03};
we also recommend \cite{Agr06} as a suitable review on characteristic connections).
We are looking again for metric connections with skew torsion and preserving the fixed
$G$-structure. If it exists, it is called the characteristic 
torsion of the fixed $G$-type and we denote by $T^c$. However, not all $G$-structures admit such a connection. The question whether or not a certain $G$-type admits a characteristic connection can be decided using representation theory. Indeed, consider 
the $G$-morphism
\bdm
\Theta : \Lambda^3(\R^n) \ \longrightarrow \ \R^n \otimes \mathfrak{m} \ , 
\quad \Theta(T) \ := \ \sum_{i=1}^n e_i \otimes  \mathrm{pr}_{\mathfrak{m}}(e_i \haken T) \ .
\edm
\begin{thm}[\cite{F03}]
A $G$-structure of a Riemannian manifold admits a characteristic connection if and only if the 
intrinsic torsion $\Gamma$ belongs to the image of $\Theta$. In this case
the characteristic torsion $T$ and the intrinsic torsion are related by the formula 
$2 \cdot \Gamma=- \Theta(T)$.
\end{thm}
\section{Gray-Hervella classification}
%
Let $(M,g,J)$ be  a $6$-dimensional almost Hermitian manifold. Then the corresponding 
$U(3)$-structure 
is given by the Lie algebra decomposition $\mathfrak{so}(6)=\mathfrak{u}(3)\oplus\mathfrak{m}^6$. 
One can directly compute the decomposition of $\R^6\otimes\mathfrak{m}^6$.
The $\mathrm{U}(3)$-representation $\R^6 \otimes \mathfrak{m}^6$ splits
into four irreducible representations,
\bdm
\R^6 \otimes \mathfrak{m}^6 \ = \ \mathcal{W}_1 \oplus \mathcal{W}_2 \oplus 
\mathcal{W}_3 \oplus \mathcal{W}_4 \ .
\edm
These are the basic classes of $U(3)$-structures in the Gray-Hervella
classification. The manifolds of type 
$\mathcal{W}_1$ are exactly the nearly  K\"ahler manifolds.
On the other side,
the $\mathrm{U}(3)$-representation $\Lambda^3(\R^6)$ splits into three irreducible
components,
\bdm
\Lambda^3(\R^6) \ = \ \mathcal{W}_1 \oplus 
\mathcal{W}_3 \oplus \mathcal{W}_4 \ .
\edm
The reader can find an explicit description of these decompositions in the paper \cite{AFS05}. 
Together, this allows us to describe more explicitly 
the $\mathrm{U}(3)$-structures admitting a characteristic connection:
\begin{cor}[\cite{FI02}]
A $\mathrm{U}(3)$-structure admits a characteristic connection if and only if the 
$\mathcal{W}_2$-component of the intrinsic torsion vanishes.
\end{cor}
Let us finally summarize some results of Gray and Kirichenko on nearly  K\"ahler 
manifolds in dimension $6$. 
A nearly K\"ahler manifold is said to be of constant type if there exists  a positive constant 
$\alpha$ such that for all vector fields
\bdm
\|(\nabla^g_XJ)(Y)\|^2=\alpha[\|X\|^2\|Y\|^2-g(X,Y)^2-g(JX,Y)^2].
\edm
\begin{thm}[\cite{G76}]
Let $(M,g,J)$ be a 6--dimensional nearly  K\"ahler manifold that is not  K\"ahler. Then
\begin{itemize}
\item[ii)] $M$ is of constant type,
\item[iii)] $g$ is an Einstein  metric on $M$,
\item[iv)] the first Chern class of $M$ vanishes.
\end{itemize}
\end{thm}
\begin{thm}[\cite{Kir77}]
The characteristic torsion of a nearly K\"ahler $6$-manifold is parallel
with respect to the characteristic connection, $\nabla^c T^c = 0$.
\end{thm}
\section{Spinorial approach}
%
There is another characterization of $6$-dimensional nearly K\"ahler manifolds due to 
R. Grunewald, involving the existence of so called 
real Killing spinors (\cite{Gru90}, see also \cite{BFGK91}). 
Let us first introduce basic facts and definitions.
The real Clifford algebra in dimensions $6$ is isomorphic to 
End$(\R^8)$. The spin 
representation is real, $8$-dimensional and we denote it by 
$\Delta := \mathbb{R}^8$. By fixing an orthonormal basis $e_1, \ldots ,e_6$ of the 
Euclidean  space $\mathbb{R}^6$, one choice for the real representation of 
the Clifford algebra on $\Delta$ is \label{SU3basis}
\begin{align*}
 e_{1} &=  +E_{18} + E_{27} - E_{36} - E_{45}, \qquad
e_{2} =-E_{17} + E_{28} + E_{35} - E_{46},\\
e_{3} &= -E_{16} + E_{25} - E_{38} + E_{47}, \qquad
e_{4} = -E_{15} - E_{26} - E_{37} - E_{48},\\
e_{5} &= -E_{13} - E_{24} + E_{57} + E_{68}, \qquad
e_{6} =+E_{14} - E_{23} - E_{58} + E_{67},
\end{align*}
where the matrices $E_{ij}$ denote the standard basis elements of the 
Lie algebra 
$\mathfrak{so}(8)$, i.\,e. the endomorphisms mapping $e_i$ to $e_j$, $e_j$ to $-e_i$ and 
everything else to zero. The 
spin representation 
 admits a $Spin(6)$-invariant complex structure 
$J : \Delta \rightarrow \Delta$ defined be the formula
\bdm
J \ := \ e_1 \cdot e_2 \cdot e_3 \cdot e_4 \cdot e_5 \cdot e_6  .
\edm 
Indeed, $J^2 = -1$ and $J$ anti-commutes with the Clifford 
multiplication $X \cdot \phi$ by vectors $X \in\R^6$ and spinors
$\phi \in \Delta$;
this reflects the fact that $\mathrm{Spin}(6)$ is isomorphic to $\mathrm{SU}(4)$.
The complexification of $\Delta$ splits,
\bdm
\Delta \otimes_{\R} \C \ = \ \Delta^+ \, \oplus \Delta^- ,
\edm 
which is a consequence of the fact that $J$ is a real structure making $(\Delta,J)$ 
 complex-(anti)-linearly isomorphic to either $\Delta^\pm$, via 
$\phi \ \rightarrow \phi \, \pm \ i \cdot J(\phi)$. Furthermore,
any real spinor $0 \not= \phi \in \Delta$  decomposes  $\Delta$ into three pieces,
\bdm
\Delta \ = \ \R\phi  \oplus  \R\, J(\phi) \oplus  
\{ X \cdot \phi \, : \, X \in \R^6 \} .
\edm
In particular, $J$ preserves the subspaces $\{ X \cdot \phi \, : \, X \in \R^6 \}
\subset \Delta$, and the formula
\bdm
J_{\phi}(X) \cdot \phi \ := \ J(X \cdot \phi)
\edm
defines an orthogonal complex structure $J_{\phi}$ on $\R^6$ 
that depends on the spinor $\phi$. 
Moreover, the spinor determines a $3$-form by means of
\bdm
\omega_{\phi}(X,Y,Z) \ := \ -(X \cdot Y \cdot Z \cdot \phi \, , \, \phi)
\edm
where the brackets indicate the inner product on $\Delta$.
The pair $(J_{\phi} , \omega_{\phi})$ is an $\mathrm{SU}(3)$-structure on $\R^6$, and 
any such arises in this fashion from some real spinor. 
All this can be summarized in the known fact
that $\mathrm{SU}(3)$-structures on $\R^6$ correspond one-to-one with
real spinors of length one ($\bmod \,\Z_2$),
\bdm
\mathrm{SO}(6)/\mathrm{SU}(3)  
 \ = \ \mathbb{P}(\Delta) \ = \ \mathbb{RP}^7 \, .
\edm
These formulas proves the following 
\begin{prop}
Let $M$ be a simply connected, $6$-dimensional  Riemannian spin manifold. Then the
$\mathrm{SU}(3)$-structures on M correspond to the real spinor fields of length one defined on M.
\end{prop}
The different types of $\mathrm{SU}(3)$-structures in the sense of Gray-Hervella can be 
characterized by certain spinoral field equation for the defining spinor
$\phi$. The first result of this type has been obtained by R. Grunewald in 1990. A spinor field
$\phi$ defined on a Riemannian spin manifold is called a real Killing spinor if it satisfies
the following first order differential equation
\bdm
\nabla^g_X \phi \ = \ \lambda \cdot X \cdot \phi \, , \quad \lambda = \mathrm{const} \in \R \, .
\edm
If $\lambda = 0$, the spinor field is simply parallel. 
Real Killing spinors are the eigenspinors of the Dirac operator realizing the
lower bound of the Dirac spectrum given by Th. Friedrich in 1980, see \cite{F80}. 
Now we can formulate the mentioned result:
\begin{thm}(see \cite{Gru90})
Let $(M,g)$ be a $6$-dimensional spin manifold admitting a non-trivial 
real Killing spinor. Then $M$ is nearly  K\"ahler.
Conversely, any simply connected nearly K\"ahler $6$-manifold admits non-trivial 
Killing spinor.
\end{thm} 
\noindent
We sketch the proof of the first statement. Suppose that $\phi$ is a Killing spinor,
$\nabla^g_X \phi= X \cdot \phi$. We differentiate the equation
$J_{\phi}(X) \cdot \phi = J(X \cdot \phi)$ :
\bdm
\nabla^g_Y(J_{\phi}(X))\cdot \phi + J_{\phi}(X) \cdot \nabla^g_Y \phi \, = \,
J(\nabla^g_YX \cdot \phi) + J(X \cdot \nabla^g_Y \phi)  \, = \,
J_{\phi}(\nabla^g_YX) \cdot \phi + J(X \cdot \nabla^g_Y \phi) \ .
\edm    
This formula yields the derivative $\nabla^gJ_{\phi}$ :
\bdm
(\nabla^g_Y J_{\phi})(X) \cdot \phi \ = \ J(X \cdot \nabla^g_Y \phi) - 
J_{\phi}(X) \cdot \nabla^g_Y \phi \ = \ J(X\cdot Y \cdot \phi) - 
J_{\phi}(X) \cdot (Y \cdot \phi) \ .
\edm
In particular, for $X = Y$ we obtain
\begin{eqnarray*}
(\nabla^g_X J_{\phi})(X) \cdot \phi &=& - \|X \|^2 J(\phi) -
J_{\phi}(X) \cdot X \cdot \phi \ = \ - \|X \|^2 J(\phi) +
X \cdot J_{\phi}(X) \cdot \phi \\ 
&=& - \|X \|^2 J(\phi) + X \cdot J(X \cdot \phi) \ = \
- \|X \|^2 J(\phi) - X \cdot X \cdot J(\phi) \ = \ 0 \ . 
\end{eqnarray*}
Finally, the almost complex structure $J_{\phi}$ is nearly K\"ahler.
\begin{remark}
The spinor field equations for all other types of $\mathrm{SU}(3)$-structures have 
been discussed in the paper \cite{ACFH15} .
\end{remark}
\begin{example}
The $6$-dimensional sphere admits real Killing spinors. Indeed, fix a
constant spinor in the Euclidean space $\R^7$ and restrict it to $S^6$.
Then it becomes a real Killing spinor on the sphere. Moreover, this spinor
defines its standard nearly K\"ahler structure described before.  
\end{example}
\section{Non-homogeneous nearly K\"ahler manifolds}
%
Although it had been widely believed that non-homogeneous nearly K\"ahler manifolds
should exist, their explicit construction was an open problem for many years,
in contrary to
their 
odd-dimensional siblings, nearly parallel $G_2$-manifolds, had been much less reluctant to
provide inhomogeneous examples. On the path to a solution, several approaches had been tried 
that provided new insights into the shape and properties of nearly K\"ahler manifolds, but had
not brought the answer to the original problem. For example, nearly hypo structures allow the
construction of compact nearly K\"ahler structures with conical singularities \cite{FIMU08},
and infinitesimal deformations of nearly K\"ahler structures lead to interesting spectral problems
on Laplacians \cite{MS11}. Local homogeneous non-homogeneous examples of nearly K\"ahler
manifolds were described in \cite{CV15}.

The main breakthrough was obtained very recently by Foscolo and Haskins \cite{FH17}, which we shall now
shortly describe as it relates directly to our object of investigation, $S^6$:
\begin{thm}[Foscolo, Haskins]\label{fh}
There exists a non-homogeneous nearly  K\"ahler structure on $S^6$ and on $S^3\times S^3$.
\end{thm}
These are the first example of non-homogeneous compact  nearly  K\"ahler $6$-manifolds. 
Recall that Butruille \cite{B05} showed that the only homogeneous compact nearly  K\"ahler 
$6$-manifolds are $S^6$, $S^3\times S^3$, $\CP^3$ and the flag manifold $F_2$. The 
examples of L. Foscolo and M. Haskins are based on weakening of the assumption of homogeneity: 
they are cohomogeneity one, i.e., they admit an isometric action of a compact Lie group 
 such that generic orbits of the action are of codimension one. The Lie group considered in 
this case is $\mathrm{SU}(2)\times \mathrm{SU}(2)$ and the generic orbits are $S^2\times S^3$ which is 
motivated by results of Podesta and Spiro \cite{PS12} characterizing all possible groups 
and orbits for cohomogeneity one nearly  K\"ahler. In fact  L. Foscolo and M. Haskins 
state the following conjecture.
\begin{con}
The only simply connected cohomogeneity one compact nearly K\"ahler manifolds in dimension 
$6$ are the structures found in \cite{FH17} on $S^6$ and  $S^3\times S^3$.
\end{con}
For proof of Theorem \ref{fh} they use another, equivalent (see for example \cite{R93}) 
description of nearly K\"ahler $6$-manifolds.
\begin{prop}\label{nKe}
A $6$-dimensional manifold $(M,g,J)$ is nearly  K\"ahler if and only if there exists a 
three holomorphic form $\omega\in \Lambda^{3,0}$ and a constant $a$ such that  the following conditions hold
\bdm
d\Omega \ = \ 12 \, a\,  \,\emph{Re}(\omega) \ , \quad
d\emph{Im}(\omega) \ = \ a \, \Omega\wedge\Omega \ ,
\edm
where $\Omega=g(J\cdot,\cdot)$ is the K\"ahler form.
\end{prop}
This approach can be used to make explicit relation between nearly  K\"ahler $6$-manifolds and 
manifolds with $G_2$ holonomy which could have been suggested by the construction of the structure 
on $S^6$ from imaginary octonions. To see this, consider a $7$-dimensional Riemannian cone 
$\mathcal{C}(M)$ over a smooth compact $6$-manifold $M$ and assume that the holonomy of 
$\mathcal{C}(M)$ is contained in $G_2$. Then, $\mathcal{C}(M)$ is equipped with a $G_2$ structure, 
i.e., a $3$-form $\varphi$ and its Hodge dual $*\varphi$ with special properties. On the level 
$1$ of the cone (which can be identified with $M$) $\varphi$ and $*\varphi$ induce $\mathrm{SU}(3)$ 
structure $(\omega,\Omega)$ satisfying nearly  K\"ahler conditions from Proposition \ref{nKe}.

The main idea of the Foscolo and Haskings's proof
is to consider so-called nearly hypo structures which are the 
$\mathrm{SU}(2)$ structures induced on oriented hypersurfaces of nearly  K\"ahler $6$-manifolds 
from $\mathrm{SU}(3)$-structures. They describe the space of nearly hypo structures on $S^2\times S^3$ 
invariant under $\mathrm{SU}(2)\times \mathrm{SU}(2)$ 
action showing that it is a smooth connected $4$-manifold. 
Away from singular orbits, cohomogeneity one nearly  K\"ahler manifolds correspond to 
curves on this space satisfying some ODE equations. It turns out that there is a $2$-parameter 
family of solutions of the ODE, and to finish the proof they found conditions under which
 the solutions extend to compact nearly  K\"ahler $6$-manifold. It is important to note that 
this is closely related with desingularizations of Calabi-Yau conifold.
\end{document}